\documentclass[12pt, reqno, twoside, letterpaper]{amsart}

\usepackage{z-Paperstyle}
\usepackage{graphicx}
\usepackage{tikz-cd}

\DeclareMathOperator{\Nm}{\cN_{K/\Q}}
\DeclareMathOperator{\Gal}{Gal}
\DeclareMathOperator{\re}{Re}
\DeclareMathOperator{\conv}{conv}
\DeclareMathOperator{\dist}{dist}
\DeclareMathOperator{\spn}{span}

\newcommand{\leg}[2]{\Bigl(\frac{#1}{#2}\Bigr)}
\newcommand{\smallleg}[2]{\bigl(\frac{#1}{#2}\bigr)}
\newcommand{\fd}{\mathfrak{d}}
\newcommand{\vfd}{\boldsymbol{\mathfrak{d}}}
\newcommand{\lmap}{\Theta}
\newcommand{\dbN}{\llbracket N\rrbracket}
\newcommand{\cube}{\dbN^4}
\newcommand{\ipol}{\wp}
\newcommand{\row}{\mathbf{row}}
\newcommand{\dotp}{%
  \mathbin{\vcenter{\hbox{\scalebox{0.55}{$\bullet$}}}}}

\newcommand{\ve}{\boldsymbol{e}}
\newcommand{\vh}{\boldsymbol{h}}
\newcommand{\vm}{\boldsymbol{m}}
\newcommand{\vn}{\boldsymbol{n}}
\newcommand{\vr}{\boldsymbol{r}}
\newcommand{\vs}{\boldsymbol{s}}
\newcommand{\vu}{\boldsymbol{u}}
\newcommand{\vv}{\boldsymbol{v}}
\newcommand{\vw}{\boldsymbol{w}}
\newcommand{\vx}{\boldsymbol{x}}
\newcommand{\vy}{\boldsymbol{y}}
\newcommand{\vz}{\boldsymbol{z}}
\newcommand{\valpha}{\boldsymbol{\alpha}}
\newcommand{\vbeta}{\boldsymbol{\beta}}
\newcommand{\vxi}{\boldsymbol{\xi}}
\newcommand{\vpsi}{\boldsymbol{\psi}}
\newcommand{\vzero}{\boldsymbol{0}}

\newcommand{\dnode}[2]{\begin{tabular}{@{}c@{}}#1\tabularnewline
{\footnotesize\itshape #2}\end{tabular}}
\newsavebox{\diagrambox}

\makeatletter
\patchcmd{\@settitle}{\uppercasenonmath\@title}{}{}{}
\makeatother

\title[On the OpenAI proof that Landau--Siegel zeros do not exist]
{\large On the OpenAI proof that\\ Landau--Siegel zeros do not exist}

\author{William D. Banks}
\address{Department of Mathematics\\
University of Missouri\\
Columbia, MO 65211\\
USA}
\email{bankswd@missouri.edu}


\date{October 11, 2026}

\begin{document}

\begin{abstract}
A recent manuscript of OpenAI asserts that Landau--Siegel zeros do
not exist. Here, we give a detailed account of its argument, arranged
as a direct proof with explicit constants. There are absolute constants
$q_0$ and $\delta_0>0$ such that, for every primitive real Dirichlet
character $\chi$ of conductor $q\ge q_0$, every real zero
$\beta\in(0,1)$ of $L(s,\chi)$ satisfies
$(1-\beta)\log q\ge\delta_0$, and one can take
\[
q_0=10^{10}
\mand
\delta_0=3.0\times 10^{-3}.
\]
Combined with a recent computation of Lu, Zaman, and Zhao, this
result yields the bound $(1-\beta)\log q\ge 3.0\times 10^{-3}$
for every conductor $q\ge 3$.
The proof is an interpolation determinant argument in the biquadratic
field generated by $\sqrt{2}$ over the quadratic field attached to $\chi$.
A real zero close to one forces \mbox{$\chi(p)=-1$} for most primes $p$,
and the Frobenius congruence at these primes makes a certain nonzero
determinant divisible by higher powers of these primes than its
Archimedean size allows. The underlying
ideas are those of the OpenAI manuscript. The explicit constants,
and the presentation, are new.
\end{abstract}

\maketitle

\tableofcontents

\newpage\section{Introduction and statement of results}
\label{sec:intro}

\subsection{Background}

Let $\chi$ be a primitive real Dirichlet character of conductor $q$.
For a suitable absolute constant $c_1>0$, the function $L(s,\chi)$
has at most one zero $\beta$ in the interval
$1-c_1/\log q\le\beta<1$, and such a zero is simple;
see~\cite[Ch.~14]{Dav}. A zero of this kind is called a
\emph{Landau--Siegel zero}. Whether such zeros exist is one of the
central problems of analytic number theory, since they obstruct
uniform estimates for primes in arithmetic progressions and
effective lower bounds for class numbers of quadratic fields.

Siegel~\cite{Sie} proved that $L(1,\chi)\gg_\eps q^{-\eps}$ for
every $\eps>0$, and hence $1-\beta\gg_\eps q^{-\eps}$, but his
method gives no means of computing the implied constant. Page~\cite{Pag}
showed that, for every $Q\ge 3$, at most one primitive real
character of conductor at most $Q$ has a zero in the interval
$(1-c_2/\log Q,1)$, where $c_2>0$ is an absolute constant. On the
effective side, the work of Goldfeld~\cite{Gol} and of Gross and
Zagier~\cite{GZ} yields lower bounds of logarithmic order for the
class numbers of imaginary quadratic fields, which is far from what
the nonexistence of Landau--Siegel zeros would give.

The manuscript~\cite{OAI} of OpenAI, dated October 1, 2026, asserts
that Landau--Siegel zeros do not exist. More precisely, it states
that there is an absolute constant $c>0$ such that
$(1-\beta)\log q\ge c$ for every real zero $\beta\in(0,1)$ of
$L(s,\chi)$. The proof given there is quite short. It is an
interpolation determinant argument of a kind familiar from transcendence
theory~\cite{Lau}, and its only analytic input is a classical
inequality for the logarithmic derivative of $L(s,\chi)$.

The argument of the present paper is the argument of~\cite{OAI}, and
we claim no originality for its ideas. Our aim is to make that
argument more accessible and to confirm it. The proof in~\cite{OAI}
is terse, and it proceeds by contradiction along a sequence of
characters, so it provides no numerical value for any constant. Here
the proof is arranged as a direct argument in which every step is
written out, and the notation has been chosen with the reader in mind.
The explicit constants $q_0$ and $\delta_0$ of Theorem~\ref{thm:main}
are new.

The paper was prepared as follows. The author first studied the
manuscript~\cite{OAI}, and, with the assistance of the large language
model Claude (Anthropic), produced an initial report on the manuscript
and an expanded draft of the proof. That draft is the
basis of the present text. The author worked through the
argument line by line, rewriting proofs and changing notation
to aid understanding. Every line of the paper has been
verified by the author, who believes the argument to be sound and
takes full responsibility for the present text. In view of the importance of
the statement, independent scrutiny is essential; this paper is
written to make that scrutiny as easy as possible.

\subsection{Results}

In this paper we give a detailed proof that Landau--Siegel zeros do
not exist, in a form that is explicit for all large conductors. The
statement is that of~\cite{OAI} (where no constant is given a
numerical value) together with admissible values of the two
constants that govern large conductors. Our main result is the following.

\begin{theorem}\label{thm:main}
There are absolute constants $q_0\ge 9$ and $\delta_0>0$ such that for
every primitive real Dirichlet character $\chi$ of conductor
$q\ge q_0$, every real zero $\beta\in(0,1)$ of $L(s,\chi)$ satisfies
\[
(1-\beta)\log q\ge\delta_0.
\]
One can take $q_0=10^{10}$ and
$\delta_0=3.0\times 10^{-3}$. Consequently, there is an absolute
constant $c>0$ such that $(1-\beta)\log q\ge c$ for every real zero
$\beta\in(0,1)$ of $L(s,\chi)$, for every primitive real Dirichlet
character $\chi$ of conductor $q\ge 3$. By the computation
of~\cite{LZZ}, one can take $c=3.0\times 10^{-3}$.
\end{theorem}

\begin{proof}[Proof of the second assertion from the first]
There are only finitely many primitive real characters $\chi$ of
conductor $q<q_0$. For each of them, the entire function $L(s,\chi)$
is not identically zero, so it has only finitely many zeros in
$[0,1]$. Hence $(1-\beta)\log q$ takes only
finitely many values, all of them positive, as $\chi$ runs over
these characters and $\beta$ over the real zeros of $L(s,\chi)$ in
$(0,1)$. Let $c_0$ be the least of these values, or
$c_0\defeq\delta_0$ if there are none. Then the second assertion
holds with $c\defeq\min\{\delta_0,c_0\}$.

For the last assertion, we take $q_0=10^{10}$ and
$\delta_0=3.0\times 10^{-3}$. By
\mbox{\cite[Theorem~1.1]{LZZ}}, which rests on a computation of
Lu, Zaman, and Zhao, the function $L(s,\chi)$ has no real zero
$\beta\ge 1-1/(5\log q)$ when $q\le 10^{10}$. Hence
$c_0>\tfrac15>\delta_0$, and therefore $c=\delta_0$.
\end{proof}

\subsection{Outline of the proof}
\label{sec:outline}

Throughout, $\beta\in(0,1)$ is a real zero of $L(s,\chi)$, and we put
$\delta\defeq(1-\beta)\log q$. We must show that $\delta$ cannot be
small when $q$ is large. The proof bounds a single nonzero integer
from above and from below, and it has four ingredients.

\medskip\noindent\emph{Prime bias.}
A zero with $\delta$ small forces $\chi(p)=-1$ for most primes $p$.
More precisely, Lemma~\ref{lem:bias} shows that for all
$X\ge 8$ and all integers $H\ge 319$,
\[
\ssum{H<p\le X\\ \chi(p)=-1}\frac{\log p}{p}
\ge\log X-\log H-C_1\,\frac{\delta(\log X)^2}{\log q}
-\Bigl(C_2+\frac{1}{2H}\Bigr)\log q
\]
with explicit absolute constants $C_1$ and $C_2$. The sum of $(\log p)/p$ over
all primes $p\le X$ is $\log X+O(1)$. Hence, if $X$ is a large power
of $q$ and $\delta$ is small, then the primes with $\chi(p)=-1$
carry almost all of the logarithmic mass. This inequality is
classical, and it is the only place where the zero $\beta$ enters.

\medskip\noindent\emph{The determinant.}
Let $\Q(\sqrt{d}\,)$ be the quadratic field attached to $\chi$, with
$d$ squarefree, and fix $a,b\in\C$ with $a^2=d$ and $b^2=2$. For
$q\ge 9$ the field $K\defeq\Q(a,b)$ is biquadratic, with Galois
group $\Gal(K/\Q)=\{1,\sigma,\tau,\sigma\tau\}$, where $\sigma$
changes the sign of $a$ and $\tau$ changes the sign of $b$. For an integer
$N\ge 1$ and $\vn\in\{0,\ldots,N-1\}^4$, we put
\[
\theta_{\vn}\defeq n_1+n_2a+n_3b+n_4ab,
\]
and to every $\valpha=(\alpha_1,\alpha_2,\alpha_3)\in\Z_{\ge 0}^3$
we attach the row vector
\[
\row_{\valpha}\defeq\bigl(\theta_{\vn}^{\alpha_1}\,\sigma(\theta_{\vn})^{\alpha_2}\,
\sigma\tau(\theta_{\vn})^{\alpha_3}\bigr)_{\vn},
\]
which has $N^4$ entries. Let $M\defeq N^{4/3}$, so that $M^3=N^4$.
There are about $N^4$ rows with every $\alpha_j$ of order $M$, and
Proposition~\ref{prop:interp} shows that rows of this size suffice
to span the full space, with a constant that
\emph{does not depend on $d$.}
Its proof uses only linear algebra and convex geometry. The
proposition also allows unequal degrees. For integers
$1\le H\le N$, the rows with
\[
\alpha_1\le 32H^{2/3}M\mand\alpha_2,\alpha_3\le 32H^{-1/3}M
\]
already span (Corollary~\ref{cor:box}). We order all rows by the
weight $\alpha_1+H\alpha_2+H\alpha_3$ and retain a row when it is
not in the span of the earlier ones. This gives $N^4$ linearly
independent rows (Lemma~\ref{lem:greedy}), and their determinant
$\Delta$ is a nonzero element of $\Z[a,b]$. Thus, $\Nm(\Delta)$ is a
nonzero integer. Let $S_1$ be the sum of $\alpha_1$ over the
retained rows, and let $S_{23}$ be the sum of $\alpha_2+\alpha_3$.
For a suitable fixed $H$, Lemma~\ref{lem:degree} gives
\[
S_1\ge N^4M\mand S_{23}\le\tfrac{1}{12}\,S_1,
\]
so that nearly all of the degree lies in the first variable.

\medskip\noindent\emph{The Archimedean bound.}
Every conjugate of $\theta_{\vn}$ has absolute value less than
$8N\sqrt{q}$. Hadamard's inequality, applied to each conjugate of
$\Delta$, gives (Lemma~\ref{lem:arch})
\[
\frac14\log|\Nm(\Delta)|\le(S_1+S_{23})
\Bigl(\log N+\frac12\log q+\log 8\Bigr)+2N^4\log N.
\]

\medskip\noindent\emph{The $p$-adic bound.}
Let $p$ be an odd prime with $\chi(p)=-1$. Then
$\theta^p\equiv g_p(\theta)\bmod{p\,\Z[a,b]}$ for all
$\theta\in\Z[a,b]$, where $g_p$ is either $\sigma$ or $\sigma\tau$
(Lemma~\ref{lem:frob}). The point is that $g_p(\theta_{\vn})$ is
again one of our three variables. In a retained row, write
$\alpha_1=pk+\varrho$ with $0\le\varrho<p$, and replace the factor
$\theta_{\vn}^{pk}$ by
$\bigl(\theta_{\vn}^p-g_p(\theta_{\vn})\bigr)^k$. Then every entry of
the new row is divisible by $p^k$. On the other hand, the new row
differs from the old one by a combination of rows in which $p$
factors of $\theta_{\vn}$ have been exchanged for one factor of
$g_p(\theta_{\vn})$, several times over. If $p>H$, then these rows
have smaller weight, so they lie in the span of the earlier
retained rows, and the determinant is unchanged. Therefore
$\Delta$ is divisible by $p^{E_p}$, where $E_p$ is the sum of
$\lfloor\alpha_1/p\rfloor$ over the retained rows
(Lemma~\ref{lem:extract}). Summing over the primes $p\le M$, we
obtain (Corollary~\ref{cor:padic})
\[
\frac14\log|\Nm(\Delta)|\ge
S_1\ssum{H<p\le M\\ \chi(p)=-1}\frac{\log p}{p}-2N^4M.
\]

\medskip\noindent\emph{The comparison.}
The lattice has rank four, but there are only three variables, and
this is the reason for the relation $\log N=\tfrac34\log M$. The
main term of the upper bound is therefore
\[
\tfrac34(S_1+S_{23})\log M\le\tfrac{13}{16}\,S_1\log M.
\]
If the primes with $\chi(p)=-1$ carried all of the logarithmic
mass up to $M$, then the main term of the lower bound would be
$S_1\log M$, which is larger. The prime bias makes this precise.
We take $N\defeq\lceil q^\omega\rceil$ with a large fixed integer
$\omega$, so that every term of size $\log q$ is at most
$\tfrac{3}{32}\log M$, and the terms $2N^4\log N$ and $2N^4M$ are
$O(S_1)$. Dividing by $S_1\log M$, we are left with
\[
\frac{3}{16}\le\frac43(\omega+1)C_1\,\delta+\frac{3}{32}
\]
for all $q\ge q_0$, and this provides a lower bound for $\delta$.

\bigskip
The proof of the first assertion of Theorem~\ref{thm:main} occupies
\S\S\ref{sec:analytic}--\ref{sec:conclusion}. The diagram in
\S\ref{sec:diagram} shows how the results used in it depend on one
another.

\subsection{Conventions}

Throughout, $\chi$ denotes a primitive real Dirichlet character of
conductor $q\ge 3$, and we put
\[
\sL\defeq\log q.
\]
For any real zero $\beta\in(0,1)$ of $L(s,\chi)$, we put
$\delta\defeq(1-\beta)\sL$.

The letter $p$ always denotes a prime number, and $\Lambda$ is the
von Mangoldt function. We adopt the
convention that $z^0=1$ for all $z\in\C$, including $z=0$.

We write $|S|$ for the cardinality of a finite set $S$. Vectors are
set in boldface and their coordinates in the
corresponding lightface letter; thus $\vx=(x_1,\ldots,x_4)$ for
$\vx\in\R^4$, and $\vzero$ is the zero vector. For
$\vx,\vy\in\R^4$ we denote by
\[
\vx\dotp\vy\defeq x_1y_1+\cdots+x_4y_4
\mand
\|\vx\|\defeq(\vx\dotp\vx)^{1/2}
\]
the dot product and the Euclidean norm, respectively.
We use the symbols $\ll$, $\gg$, and $O$ with their usual meanings.
The implied constants are absolute, as is every constant
introduced below.

\subsection{The quadratic field attached to
\texorpdfstring{$\chi$}{chi}}
\label{sec:quadfield}

A primitive real character of conductor $q$ is the Kronecker symbol
\[
\chi(n)=\Bigl(\frac{D}{n}\Bigr)
\]
for a fundamental discriminant $D$ with $|D|=q$;
see~\cite[Ch.~5]{Dav}.
From now on, we write $\Q(\sqrt{D}\,)=\Q(\sqrt{d}\,)$
with $d\ne 1$ squarefree, and so $D=d$ or $D=4d$. Then
\be\label{eq:field}
1\le|d|\le q,\qquad \chi(p)=\leg{d}{p}\quad\text{(Legendre symbol) for
every odd prime $p$.}
\ee
The case $d=2$ occurs only for $q=8$, so it is excluded when $q\ge 9$.

\subsection{Dependency diagram}\label{sec:diagram}

The following diagram shows how the results of the paper depend on
one another.

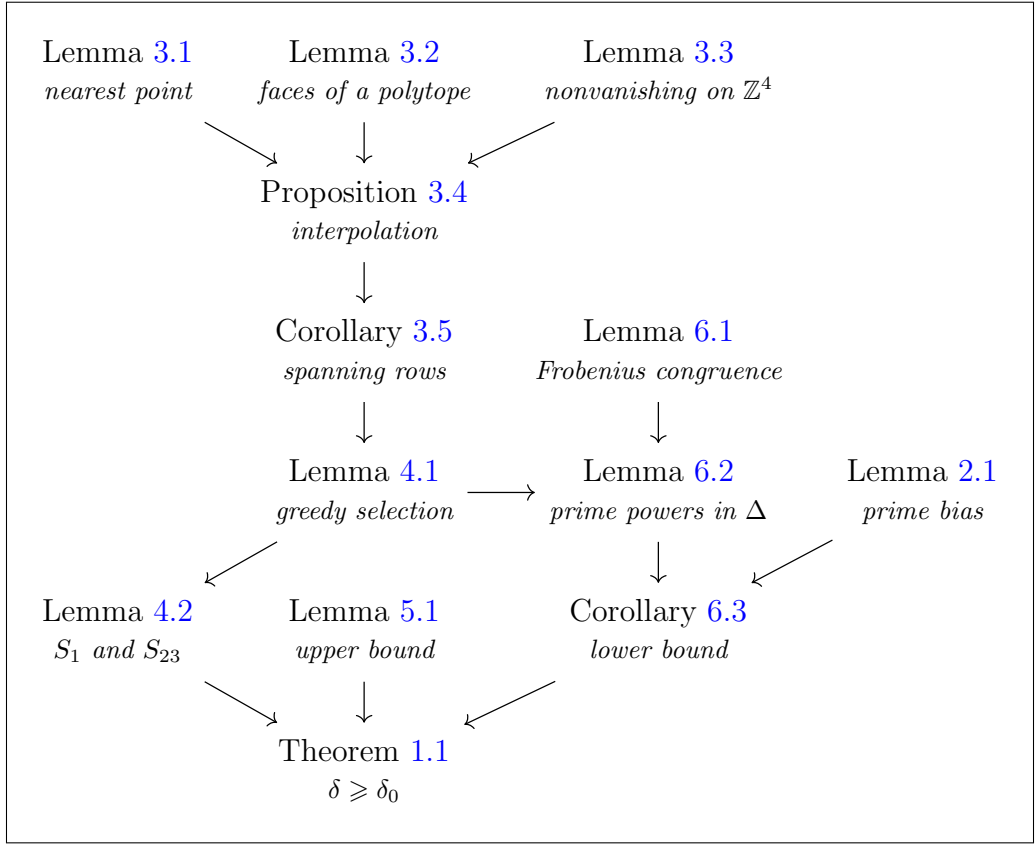
\begin{lrbox}{\diagrambox}
\begin{tikzcd}[column sep=small,row sep=scriptsize]
\dnode{Lemma~\ref{lem:nearest}}{nearest point} \arrow[dr]
& \dnode{Lemma~\ref{lem:faces}}{faces of a polytope} \arrow[d]
& \dnode{Lemma~\ref{lem:nonvanish}}{nonvanishing on $\Z^4$}
  \arrow[dl]
& \\
& \dnode{Proposition~\ref{prop:interp}}{interpolation} \arrow[d]
& & \\
& \dnode{Corollary~\ref{cor:box}}{spanning rows} \arrow[d]
& \dnode{Lemma~\ref{lem:frob}}{Frobenius congruence} \arrow[d] & \\
& \dnode{Lemma~\ref{lem:greedy}}{greedy selection}
  \arrow[dl] \arrow[r]
& \dnode{Lemma~\ref{lem:extract}}{prime powers in $\Delta$}
  \arrow[d]
& \dnode{Lemma~\ref{lem:bias}}{prime bias} \arrow[dl] \\
\dnode{Lemma~\ref{lem:degree}}{$S_1$ and $S_{23}$} \arrow[dr]
& \dnode{Lemma~\ref{lem:arch}}{upper bound} \arrow[d]
& \dnode{Corollary~\ref{cor:padic}}{lower bound} \arrow[dl] & \\
& \dnode{Theorem~\ref{thm:main}}{$\delta\ge\delta_0$} & &
\end{tikzcd}
\end{lrbox}
\[
\setlength{\fboxsep}{8pt}
\scalebox{1}{\fbox{\usebox{\diagrambox}}}
\]

\section{Analytic input}
\label{sec:analytic}

We begin with the classical bound for the logarithmic
derivative of $L(s,\chi)$ on which the prime bias rests. Let
$\beta\in(0,1)$ be a real zero of $L(s,\chi)$, and let
$\eps\in\{0,1\}$ be such that $\chi(-1)=(-1)^\eps$.
The classical partial fraction formula (see~\cite[Ch.~12,
eqs.~(17)--(18)]{Dav}) states that for all $s$,
\[
\frac{L'}{L}(s,\chi)=-\frac12\log\frac{q}{\pi}
-\frac12\,\frac{\Gamma'}{\Gamma}\Bigl(\frac{s+\eps}{2}\Bigr)
+B(\chi)+\sum_{\rho}\Bigl(\frac{1}{s-\rho}+\frac{1}{\rho}\Bigr),
\]
where $\rho$ runs over the nontrivial zeros with multiplicity (all
have $0\le\re\rho\le 1$) and $\sum_\rho\re(1/\rho)$ converges. Moreover,
\[
\re B(\chi)=-\sum_\rho\re(1/\rho).
\]
Since $\chi$ is real, $L'/L$ is real for real $s=\sigma>1$; thus,
taking real parts, we get that
\be\label{eq:logder}
-\frac{L'}{L}(\sigma,\chi)=\frac12\log\frac{q}{\pi}
+\frac12\,\frac{\Gamma'}{\Gamma}\Bigl(\frac{\sigma+\eps}{2}\Bigr)
-\sum_{\rho}\re\frac{1}{\sigma-\rho}\qquad(\sigma>1).
\ee
Because
\[
\re\frac{1}{\sigma-\rho}=\frac{\sigma-\re\rho}{|\sigma-\rho|^2}>0
\]
for every zero in the sum~\eqref{eq:logder}, discarding all zeros
except $\rho=\beta$,
it follows that
\be\label{eq:logder2}
-\frac{L'}{L}(\sigma,\chi)\le\frac12\log\frac{q}{\pi}
+\frac12\,\frac{\Gamma'}{\Gamma}\Bigl(\frac{\sigma+\eps}{2}\Bigr)
-\frac{1}{\sigma-\beta}\qquad(\sigma>1).
\ee

\begin{lemma}[Prime bias]\label{lem:bias}
Let $\lambda\in(0,2]$, and let $C_1$ and $C_2$ be real numbers such that
\[
C_1\ge\frac{\er^{\lambda}}{2\lambda^2}\mand
C_2\ge\frac{\er^{\lambda}}{4}.
\]
Let $\beta\in(0,1)$ be a real zero of $L(s,\chi)$ and
$\delta=(1-\beta)\sL$. Then for all $X\ge 8$,
\be\label{eq:bias1}
\ssum{p\le X\\ \chi(p)=1}\frac{\log p}{p}
+\frac12\ssum{p\le X\\ p\,\mid\,q}\frac{\log p}{p}
\le C_1\,\frac{\delta(\log X)^2}{\sL}+C_2\,\sL,
\ee
and for all $X\ge 8$ and all integers $H\ge 319$,
\be\label{eq:bias2}
\ssum{H<p\le X\\ \chi(p)=-1}\frac{\log p}{p}
\ge\log X-\log H-C_1\,\frac{\delta(\log X)^2}{\sL}
-\Bigl(C_2+\frac{1}{2H}\Bigr)\sL.
\ee
In particular, with $\lambda=\tfrac32$ one can take
$C_1=0.99594$ and $C_2=1.12043$.
\end{lemma}

\begin{proof}
For $\sigma\in(1,2]$ the argument $(\sigma+\eps)/2$ in
\eqref{eq:logder2} lies in $(\tfrac12,\tfrac32]$, and
$\Gamma'/\Gamma$ is increasing on the positive real axis, so
\[
-\frac{L'}{L}(\sigma,\chi)\le\frac{\sL}{2}-\kappa
-\frac{1}{\sigma-\beta}\qquad(\sigma\in(1,2]),
\]
where
\[
\kappa\defeq\frac12\log\pi-\frac12\,\frac{\Gamma'}{\Gamma}
\Bigl(\frac32\Bigr)=0.5541\cdots.
\]
Since $-\zeta'/\zeta(\sigma)\le 1/(\sigma-1)$ for $\sigma\in(1,2]$,
it follows that
\be\label{eq:sumbound}
\sum_{n\ge 1}\frac{\Lambda(n)(1+\chi(n))}{n^\sigma}
\le\frac{\sL}{2}-\kappa+\frac{1}{\sigma-1}-\frac{1}{\sigma-\beta}
\qquad(\sigma\in(1,2]).
\ee

Now let $X\ge 8$ and put $\sigma\defeq 1+\lambda/\log X$. Then
$\sigma\in(1,2]$, since $\log X\ge\log 8>2\ge\lambda$. Every term
of the series in \eqref{eq:sumbound} is nonnegative, so the series
is at least its subsum over the primes $p\le X$ with $\chi(p)\ne-1$.
For such a prime, $1+\chi(p)$ equals $2$ if $\chi(p)=1$ and equals
$1$ if $p\mid q$, and
\[
\frac{\log p}{p^\sigma}
=\frac{\log p}{p}\,\er^{-\lambda\log p/\log X}
\ge\er^{-\lambda}\,\frac{\log p}{p}.
\]
Writing
\[
S_+\defeq\ssum{p\le X\\ \chi(p)=1}\frac{\log p}{p}\mand
S_0\defeq\ssum{p\le X\\ p\,\mid\,q}\frac{\log p}{p},
\]
the left side of \eqref{eq:sumbound} is therefore at least
$\er^{-\lambda}(2S_++S_0)$. On the right side,
\[
\frac{1}{\sigma-1}-\frac{1}{\sigma-\beta}
=\frac{1-\beta}{(\sigma-1)(\sigma-\beta)}
\le\frac{1-\beta}{(\sigma-1)^2}
=\frac{\delta(\log X)^2}{\lambda^2\sL}.
\]
Multiplying by $\er^{\lambda}/2$ and using the lower bounds for
$C_1$ and $C_2$, we obtain
\be\label{eq:bias0}
S_++\frac12\,S_0\le C_1\,\frac{\delta(\log X)^2}{\sL}+C_2\,\sL
-\frac{\er^{\lambda}\kappa}{2},
\ee
which is stronger than \eqref{eq:bias1}.

We turn to \eqref{eq:bias2}. If $X\le H$, the sum is empty and the
right side of \eqref{eq:bias2} is negative, so we may assume that
$X>H\ge 319$. By the explicit form of Mertens'
theorem~\cite[Theorem~6]{RS}, with a certain absolute constant $E$,
\[
\sum_{p\le x}\frac{\log p}{p}>\log x+E-\frac{1}{2\log x}\quad(x>1),
\qquad
\sum_{p\le x}\frac{\log p}{p}<\log x+E+\frac{1}{2\log x}\quad(x\ge
319).
\]
Using the first bound at $x=X$ and the second at $x=H$, we get
\be\label{eq:mertens}
\sum_{H<p\le X}\frac{\log p}{p}
>\log X-\log H-\frac{1}{2\log X}-\frac{1}{2\log H}
>\log X-\log H-\frac{1}{\log H}.
\ee
A prime $p\in(H,X]$ with $\chi(p)\ne-1$ has $\chi(p)=1$ or
$p\mid q$. Hence
\be\label{eq:remove}
\ssum{H<p\le X\\ \chi(p)=-1}\frac{\log p}{p}
\ge\sum_{H<p\le X}\frac{\log p}{p}
-\Bigl(S_++\frac12\,S_0\Bigr)
-\frac12\ssum{p>H\\ p\,\mid\,q}\frac{\log p}{p}.
\ee
If $k$ is the number of primes $p>H$ dividing $q$, then $H^k<q$,
so $k<\sL/\log H$. Since $(\log t)/t$ is decreasing for $t\ge\er$,
each such prime contributes at most $(\log H)/H$, and therefore
\[
\ssum{p>H\\ p\,\mid\,q}\frac{\log p}{p}
\le\frac{\sL}{\log H}\cdot\frac{\log H}{H}=\frac{\sL}{H}.
\]
Inserting this bound, \eqref{eq:bias0}, and \eqref{eq:mertens}
into \eqref{eq:remove}, we find that the left side of
\eqref{eq:bias2} exceeds
\[
\log X-\log H-C_1\,\frac{\delta(\log X)^2}{\sL}
-\Bigl(C_2+\frac{1}{2H}\Bigr)\sL
+\frac{\er^{\lambda}\kappa}{2}-\frac{1}{\log H}.
\]
Finally, $\er^{\lambda}\kappa/2>\kappa/2>0.277$ and
$1/\log H\le 1/\log 319<0.174$, so the last two terms have a
positive sum, and \eqref{eq:bias2} follows.

For $\lambda=\tfrac32$ we have $\er^{3/2}=4.48168\cdots$, hence
$\er^{\lambda}/(2\lambda^2)=0.995930\cdots$ and
$\er^{\lambda}/4=1.120422\cdots$.
\end{proof}

\section{The interpolation proposition}
\label{sec:interpolation}

This section is pure linear algebra and convex geometry; there is no
number theoretic content. We begin with three elementary facts.

\begin{lemma}[Nearest point]\label{lem:nearest}
Let $\Omega\subset\R^4$ be nonempty, compact, and convex, and let
$\vz\in\R^4\setminus \Omega$. Let $\vy$ be a point in $\Omega$ that
lies
nearest $\vz$, and put $\vh\defeq\vz-\vy$. Then
\[
\vh\dotp(\vx-\vy)\le 0\qquad(\vx\in \Omega).
\]
\end{lemma}

\begin{proof}
For $\vx\in \Omega$ and $0<\eta\le 1$, the point
$\vy+\eta(\vx-\vy)$ lies in $\Omega$ by convexity, and
\[
\|\vh\|=\dist(\vz,\vy)
\le\dist(\vz,\vy+\eta(\vx-\vy))
=\bigl\|\vh-\eta(\vx-\vy)\bigr\|.
\]
Squaring both sides and expanding using $\|\vu\|^2=\vu\dotp\vu$, we get
\[
0\le -2\eta\,\vh\dotp(\vx-\vy)+\eta^2\|\vx-\vy\|^2.
\]
Now divide by $\eta$ and let $\eta\to 0$.
\end{proof}

\begin{lemma}[Faces of a polytope with given vertices]\label{lem:faces}
Let $\cS\subset\R^4$ be finite and nonempty, with $\conv(\cS)$ its
convex hull. Fix $\vh\in\R^4\setminus\{\vzero\}$, and let
\[
m(\vh)\defeq\max_{\vv\in\conv(\cS)}\vh\dotp\vv
\mand
\cG_{\vh}\defeq\bigl\{\vv\in\conv(\cS):\vh\dotp\vv=m(\vh)\bigr\}.
\]
Then $\cS\cap\cG_{\vh}\ne\varnothing$,
$\cG_{\vh}=\conv(\cS\cap\cG_{\vh})$, and every point in $\cG_{\vh}$ is
a convex combination of at most four points of $\cS\cap\cG_{\vh}$.
\end{lemma}

\begin{proof}
As $\cS$ is finite and nonempty, $\conv(\cS)$ is compact and
nonempty, so the maximum $m(\vh)$ is attained and
$\cG_{\vh}\ne\varnothing$. Let $\vv\in\cG_{\vh}$, and write
\[
\vv=\sum_{\vs\in\cS}\lambda_{\vs}\vs
\qquad\text{with}\quad\lambda_{\vs}\ge 0,
\quad\sum_{\vs\in\cS}\lambda_{\vs}=1.
\]
Since $\vh\dotp\vs\le m(\vh)$ for every $\vs\in\cS$, we have
\[
m(\vh)=\vh\dotp\vv=\sum_{\vs\in\cS}\lambda_{\vs}\,\vh\dotp\vs
\le m(\vh),
\]
and therefore $\vh\dotp\vs=m(\vh)$ whenever $\lambda_{\vs}>0$. Thus
every $\vs$ with $\lambda_{\vs}>0$ lies in $\cS\cap\cG_{\vh}$, which
shows that $\cS\cap\cG_{\vh}\ne\varnothing$ and that
$\vv\in\conv(\cS\cap\cG_{\vh})$. Thus,
$\cG_{\vh}\subset\conv(\cS\cap\cG_{\vh})$.
Conversely, $\cG_{\vh}$ is convex and contains $\cS\cap\cG_{\vh}$, so
$\conv(\cS\cap\cG_{\vh})\subset\cG_{\vh}$.

Finally, since $\vh\ne\vzero$, the set $\cS\cap\cG_{\vh}$ lies in
the hyperplane $\{\vx\in\R^4:\vh\dotp\vx=m(\vh)\}$, an affine space of dimension
three. By Carath\'eodory's theorem applied in that space, every point
of $\conv(\cS\cap\cG_{\vh})=\cG_{\vh}$ is a convex combination of at
most four points of $\cS\cap\cG_{\vh}$.
\end{proof}

\begin{lemma}\label{lem:nonvanish}
A nonzero polynomial $F\in\C[z_1,\ldots,z_4]$ does not vanish
identically on $\Z^4$.
\end{lemma}

\begin{proof}
Exercise.
\end{proof}

\bigskip

We now turn to the main result of this section. Given $N^4$ points
of $\C^3$ of a special shape, it produces a polynomial of small
degree that takes any prescribed values at those points. To state it,
we introduce the following notation.

For $\vfd=(\fd_1,\fd_2,\fd_3)$ with integers $\fd_1,\fd_2,\fd_3\ge 0$ let
\[
{\tt Poly}(\vfd)\defeq
\{\ipol\in\C[z_1,z_2,z_3]:\deg_{z_j}(\ipol)\le \fd_j\ (j=1,2,3)\},
\]
for which we have
\[
\dim{\tt Poly}(\vfd)=\prod_j(\fd_j+1),
\]
and for $N\ge 1$, write $\dbN\defeq\{0,1,\ldots,N-1\}$, so that
$\cube\subset\Z^4$.

\begin{proposition}[Interpolation]\label{prop:interp}
Let $\lmap:\C^4\to\C^3$ be a surjective linear map whose kernel is
$\C\vxi$, where the four coordinates of $\vxi\in\C^4$ are linearly
independent over $\Q$. Let $N\ge 1$ and let $\fd_1,\fd_2,\fd_3$ be
integers with
\be\label{eq:degcond}
\fd_j\ge 3(N-1)\quad(j=1,2,3)\mand
\prod_{j=1}^{3}\bigl(\fd_j+1-3(N-1)\bigr)>(4N-3)^4.
\ee
Then for every function $f:\cube\to\C$ there is
$\ipol\in{\tt Poly}(\vfd)$ such that
\[
\ipol(\lmap\vn)=f(\vn)\qquad(\vn\in\cube).
\]
\end{proposition}

In other words, for the given map $\lmap$, an arbitrary function $f$
on $\cube$ factors through $\lmap$ by way of a polynomial
$\ipol\in{\tt Poly}(\vfd)$, so that the following diagram commutes.
\[
\begin{tikzcd}[column sep=large,row sep=large]
\cube \arrow[r,hook] \arrow[drr,"f"',bend right=15]
& \C^4 \arrow[r,"\lmap"]
& \C^3 \arrow[d,"\ipol\,\in\,{\tt Poly}(\vfd)"] \\
& & \C
\end{tikzcd}
\]

Note that the hypothesis on $\vxi$ implies that $\lmap$ is injective on
$\Z^4$. Indeed, if $\lmap\vn=\lmap\vn'$, then
$\vn-\vn'=c\,\vxi$ for some $c\in\C$. If $c\ne 0$, then the
coordinates of $\vxi$ all lie in $c^{-1}\Z$, so they are linearly
dependent over $\Q$, which is not the case; thus, $\vn-\vn'=\vzero$.
In particular, the $N^4$ points $\lmap\vn$ with $\vn\in\cube$ are distinct.

The strength of Proposition~\ref{prop:interp} lies in the degrees. The condition
\eqref{eq:degcond} involves neither $\lmap$ nor $\vxi$, and
it is satisfied by various spaces ${\tt Poly}(\vfd)$ of dimension $O(N^4)$, which
is the least possible order since $\C^{\cube}\defeq\{f:\cube\to\C\}$
has dimension $N^4$.

\medskip

\begin{proof}[Proof of Proposition~\ref{prop:interp}]
Suppose on the contrary that the evaluation map
\[
{\tt Poly}(\vfd)\to\C^{\cube},\qquad
\ipol\mapsto\bigl(\ipol(\lmap\vn)\bigr)_{\vn},
\]
is not surjective. Since its image is then a proper
subspace, there exists a nonzero linear functional that vanishes on it.
In other words, for some $a_{\vn}\in\C$, not all zero,
\be\label{eq:annih}
\sum_{\vn\in \cube}a_{\vn}\,\ipol(\lmap\vn)
=0\qquad\text{for all $\ipol\in{\tt Poly}(\vfd)$.}
\ee
Thus, our aim is to construct a polynomial
$\ipol\in{\tt Poly}(\vfd)$ violating \eqref{eq:annih}.

Let $\cS$ be the finite set
\[
\cS\defeq\{\vn\in \cube:a_{\vn}\ne 0\}\ne\varnothing,
\]
and put
\[
\Omega\defeq\conv(\cS)\subset[0,N-1]^4,\qquad
\Omega^\sharp\defeq 4\Omega=\{4\vz:\vz\in \Omega\}.
\]

\medskip\noindent{\sc Step 1: \it an auxiliary polynomial
vanishing on $\lmap(\Omega^\sharp\cap\Z^4)$.}
Denote
\[
\fd'_j\defeq\fd_j-3(N-1)\ge 0
\mand
\vfd'\defeq(\fd'_1,\fd'_2,\fd'_3).
\]
Since
$\Omega^\sharp\subset[0,4(N-1)]^4$, we have by \eqref{eq:degcond}
\[
|\Omega^\sharp\cap\Z^4|\le(4N-3)^4<\prod_{j=1}^3(\fd'_j+1)=\dim{\tt Poly}(\vfd').
\]
The evaluation map
${\tt Poly}(\vfd')\to\C^{\Omega^\sharp\cap\Z^4}$ given by
$F\mapsto\bigl(F(\lmap\vn)\bigr)_{\vn}$ is linear, and the dimension
of its domain exceeds that of its target. The kernel is therefore
nonzero, i.e., there exists
\be\label{eq:aux}
F\in{\tt Poly}(\vfd')\setminus\{0\}\quad\text{with}\quad F(\lmap\vn)=0
\quad\text{for all $\vn\in \Omega^\sharp\cap\Z^4$.}
\ee

\medskip\noindent{\sc Step 2: \it the nearest lattice point
where $F\circ \lmap$ does not vanish.}
Since $\lmap$ is surjective, $F\circ \lmap$ is a nonzero polynomial on $\C^4$
(choose $\vz\in\C^3$ with $F(\vz)\ne 0$ and $\vx\in\C^4$ with $\lmap\vx=\vz$). By
Lemma~\ref{lem:nonvanish} the set $\{\vn\in\Z^4:F(\lmap\vn)\ne 0\}$ is
nonempty. Any large ball contains some of these lattice points, but
only finitely many, so among these choose $\vm$ for which
$\dist(\vm,\Omega^\sharp)$ is minimal. By
\eqref{eq:aux}, $\vm\notin \Omega^\sharp$. Let $\vy\in \Omega^\sharp$
be nearest to $\vm$ and put $\vh\defeq\vm-\vy\ne\vzero$, so that
$\dist(\vm,\Omega^\sharp)=\|\vh\|$. By Lemma~\ref{lem:nearest}, we have
\be\label{eq:nearest}
\vh\dotp(\vx-\vy)\le 0\qquad(\vx\in \Omega^\sharp).
\ee

\medskip\noindent{\sc Step 3: \it a supporting face and a
distinguished point $\vu$.}
Let $m(\vh)$ and $\cG_{\vh}$ be defined as in Lemma~\ref{lem:faces};
since
$\Omega=\conv(\cS)$, this means that
\[
m(\vh)=\max_{\vv\in \Omega}\vh\dotp\vv
\mand
\cG_{\vh}=\bigl\{\vv\in \Omega:\vh\dotp\vv=m(\vh)\bigr\}.
\]
Applying \eqref{eq:nearest} with $\vx=4\vv$, $\vv\in \Omega$, gives
$\vh\dotp\vv\le \vh\dotp(\tfrac14\vy)$; as $\tfrac14\vy\in \Omega$
this shows $\tfrac14\vy\in \cG_{\vh}$. By Lemma~\ref{lem:faces},
\[
\tfrac14\vy=\sum_{j=1}^{k}\lambda_j\vu_j,\quad\text{with some}\quad
k\le 4,\quad\vu_j\in \cS\cap \cG_{\vh},
\quad\lambda_j>0,\quad\sum_{j=1}^{k}\lambda_j=1.
\]
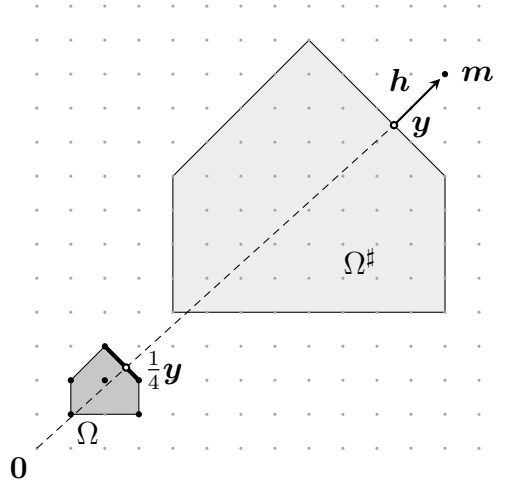
\begin{figure}[ht]
\centering
\begin{tikzpicture}[scale=0.45,>=stealth]
\filldraw[fill=black!7,draw=black]
  (4,4)--(12,4)--(12,8)--(8,12)--(4,8)--cycle;
\filldraw[fill=black!22,draw=black]
  (1,1)--(3,1)--(3,2)--(2,3)--(1,2)--cycle;
\foreach \i in {0,...,14}
  \foreach \j in {0,...,13}
    \fill[black!35] (\i,\j) circle (1.3pt);
\foreach \p in {(1,1),(3,1),(3,2),(2,3),(1,2),(2,2)}
  \fill[black] \p circle (2.6pt);
\draw[black,line width=1.6pt] (3,2)--(2,3);
\draw[black,densely dashed] (0,0)--(10.5,9.5);
\draw[black,->,thick,shorten >=2.5pt] (10.5,9.5)--(12,11);
\fill[black] (12,11) circle (2.6pt);
\filldraw[fill=white,draw=black,thick] (10.5,9.5) circle (2.6pt);
\filldraw[fill=white,draw=black,thick] (2.625,2.375) circle (2.6pt);
\node[black,right=2pt] at (12,11) {$\vm$};
\node[black,above left=-1pt] at (11.25,10.25) {$\vh$};
\node[black,below right=-1pt] at (10.75,10.0) {$\vy$};
\node[black,right=3pt] at (2.625,2.3) {$\tfrac14\vy$};
\node[black,below left=-1pt] at (0,0) {$\vzero$};
\node[black] at (1.5,0.45) {$\Omega$};
\node[black] at (9.5,5.5) {$\Omega^\sharp$};
\end{tikzpicture}
\end{figure}

For some $j$ we have $\lambda_j\ge 1/k\ge 1/4$; call the corresponding point
$\vu$. Then
\be\label{eq:upoint}
\vu\in \cS\cap \cG_{\vh},\qquad \vy-\vu\in 3\Omega.
\ee
Indeed, $\vy-\vu=\sum_j 4\lambda_j\vu_j-\vu$ is a linear combination of the
$\vu_j\in \Omega$ with coefficients $4\lambda_j$ (and $4\lambda_j-1\ge 0$
for the index of $\vu$) that are all nonnegative and sum to three,
and $\Omega$ is convex.

\medskip\noindent{\sc Step 4: \it $F\circ \lmap$ vanishes
at $\vm-\vu+\vn$ for every $\vn\in \cS\setminus \cG_{\vh}$.}
Fix $\vn\in \cS\setminus \cG_{\vh}$ and put
\[
\vy^\sharp\defeq\vy-\vu+\vn,\qquad
\vm^\sharp\defeq\vm-\vu+\vn=\vy^\sharp+\vh\in\Z^4.
\]
By \eqref{eq:upoint},
$\vy^\sharp\in 3\Omega+\Omega=4\Omega=\Omega^\sharp$
(convexity of $\Omega$). As $\vu\in \cG_{\vh}$ and
$\vn\in \Omega\setminus \cG_{\vh}$, we have
\[
\vh\dotp\vu=m(\vh),\qquad
\vh\dotp\vn<m(\vh)\qquad\Longrightarrow\qquad
\vh\dotp(\vu-\vn)>0.
\]
For $0<\eta\le 1$ the point
$\vy^\sharp+\eta(\vy-\vy^\sharp)$ lies in $\Omega^\sharp$. Since
$\vy-\vy^\sharp=\vu-\vn$ and $\vm^\sharp-\vy^\sharp=\vh$,
for all sufficiently small $\eta$ we have
\dalign{
\bigl\|\vm^\sharp-\{\vy^\sharp+\eta(\vy-\vy^\sharp)\}\bigr\|^2
&=\|\vh-\eta(\vu-\vn)\|^2\\
&=\|\vh\|^2-2\eta\,\vh\dotp(\vu-\vn)+\eta^2\|\vu-\vn\|^2<\|\vh\|^2.
}
Thus,
$\dist(\vm^\sharp,\Omega^\sharp)<\|\vh\|=\dist(\vm,\Omega^\sharp)$. By
the minimal choice of $\vm$ (and by \eqref{eq:aux} in the case that
$\vm^\sharp\in \Omega^\sharp$), we conclude that $F(\lmap\vm^\sharp)=0$.
We have therefore shown that
\be\label{eq:vanish}
F\bigl(\lmap(\vm-\vu+\vn)\bigr)=0\qquad(\vn\in \cS\setminus
\cG_{\vh}).
\ee

\medskip\noindent{\sc Step 5: \it a polynomial of degree $\le 3(N-1)$
isolating $\vu$
within $\cS\cap \cG_{\vh}$.}
The vectors $\vs-\vu$, $\vs\in \cS\cap \cG_{\vh}$, are integer vectors
orthogonal to $\vh\ne\vzero$, so they span a real subspace of
dimension $\le 3$. The rank of an integer matrix is the same over $\Q$
and over $\R$, so their $\Q$-span $W$ inside $\Q^4$ has dimension
$\le 3$. Thus, there is a nonzero $\vr\in\Z^4$ with $\vr\dotp\vw=0$
for all $\vw\in W$. Put $\nu\defeq\vr\dotp\vu\in\Z$; then
\[
\cS\cap \cG_{\vh}\subset V\defeq\{\vw\in\C^4:\vr\dotp\vw=\nu\},
\]
where the dot product is extended to $\C^4$ by the same
formula, $\vr\dotp\vw\defeq r_1w_1+\cdots+r_4w_4$. Let
\[
V_0\defeq\{\vw\in\C^4:\vr\dotp\vw=0\},
\]
and note that $V=\vu+V_0$ since $\vr\dotp\vu=\nu$.

Because $\vr$ is a nonzero integer vector, and the coordinates of
$\vxi$ are linearly independent over $\Q$, $\vr\dotp\vxi\ne 0$. Hence
the linear space meets $\ker \lmap=\C\vxi$ only in $\vzero$. The restriction
$\lmap|_{V_0}:V_0\to\C^3$ is therefore an injective linear map
between two spaces of dimension three, so it is bijective. We define
\[
\vpsi(\vz)\defeq\vu+\bigl(\lmap|_{V_0}\bigr)^{-1}(\vz-\lmap\vu)
\qquad(\vz\in\C^3).
\]
The map $\vpsi$ takes values in $V=\vu+V_0$, and its coordinates
$\psi_1,\ldots,\psi_4$ are polynomials of degree $\le 1$ in
$\vz=(z_1,z_2,z_3)$. For every $\vn\in V$ we have $\vn-\vu\in V_0$,
and therefore $\vpsi(\lmap\vn)=\vu+(\vn-\vu)=\vn$. In coordinates,
\be\label{eq:phi}
\psi_j(\lmap\vn)=n_j\qquad(\vn\in V,\ 1\le j\le 4).
\ee
Fix any $j_0$ with $r_{j_0}\ne 0$ and define
\[
G(\vz)\defeq\prod_{j\ne j_0}\ \sprod{0\le k<N\\ k\ne u_j}
\frac{\psi_j(\vz)-k}{u_j-k}.
\]
This is a polynomial of total degree at most $3(N-1)$, and the
denominators are all nonzero integers. Every factor equals one at
$\vz=\lmap\vu$ by \eqref{eq:phi}.
If $\vs\in \cS\cap \cG_{\vh}$ and $\vs\ne \vu$, then $s_j\ne u_j$ for some
$j$ \emph{different from} $j_0$,
for otherwise $0=\vr\dotp(\vs-\vu)=r_{j_0}(s_{j_0}-u_{j_0})$
would force $\vs=\vu$. For such $j$, the factor with
$s_j=k\in\{0,\ldots,N-1\}\setminus\{u_j\}$ vanishes at $\vz=\lmap\vs$
by
\eqref{eq:phi}. Hence
\be\label{eq:G}
G(\lmap\vu)=1,\qquad
G(\lmap\vs)=0\quad(\vs\in(\cS\cap \cG_{\vh})\setminus\{\vu\}),
\ee
and
\be\label{eq:degG}
\deg_{z_j}(G)\le 3(N-1)\qquad (j=1,2,3).
\ee

\medskip\noindent{\sc Step 6: \it the contradiction.}
Let $\ipol(\vz)\defeq F\bigl(\vz+\lmap(\vm-\vu)\bigr)\,G(\vz)$. For each
$j$, the first factor has degree at most $\fd'_j$ in $z_j$ by
\eqref{eq:aux}, since a translation of the variables does not
increase the degree in any single variable, and $G$ has degree at
most $3(N-1)$ in $z_j$ by \eqref{eq:degG}. Therefore
\[
\deg_{z_j}(\ipol)\le\fd'_j+3(N-1)=\fd_j\qquad(j=1,2,3),
\]
which shows that $\ipol\in{\tt Poly}(\vfd)$. For $\vn\in \cS$
we have $\ipol(\lmap\vn)=G(\lmap\vn)\,F(\lmap(\vm-\vu+\vn))$, which
vanishes for $\vn\in \cS\setminus \cG_{\vh}$ by \eqref{eq:vanish}, vanishes for
$\vn\in(\cS\cap \cG_{\vh})\setminus\{\vu\}$ by \eqref{eq:G}, and equals
$F(\lmap\vm)$ for $\vn=\vu$. Since $a_{\vn}=0$ for all $\vn\notin\cS$,
\eqref{eq:annih} gives
\[
0=\sum_{\vn\in \cS}a_{\vn} \ipol(\lmap\vn)=a_{\vu}\,F(\lmap\vm)\ne 0,
\]
a contradiction.
\end{proof}

\begin{corollary}\label{cor:box}
Let $\lmap$ be as in Proposition~\ref{prop:interp}, let $1\le H\le N$ be
integers, and put
\[
M\defeq N^{4/3},\qquad T_1\defeq 32H^{2/3}M,\qquad T_2=T_3\defeq
32H^{-1/3}M.
\]
Then the vectors
\[
\bigl((\lmap\vn)^{\valpha}\bigr)_{\vn\in\cube}
\qquad(\valpha\in\Z_{\ge 0}^3,\ \alpha_j\le T_j\text{~for~}j=1,2,3)
\]
span the space $\C^{\cube}$, where we define
$\vz^{\valpha}\defeq z_1^{\alpha_1}z_2^{\alpha_2}z_3^{\alpha_3}$
for all $\vz\in\C^3$.
\end{corollary}

\begin{proof}
Put $\fd_j\defeq\lfloor T_j\rfloor$. Since $H\le N$, we have
\[
T_2=T_3=32N(N/H)^{1/3}\ge 32N
\mand T_1\ge 32N^{4/3}\ge 32N.
\]
Hence $\fd_j\ge 3(N-1)$ and $T_j-3(N-1)\ge T_j/2$ for each $j$, so
\[
\prod_j\bigl(\fd_j+1-3(N-1)\bigr)>\prod_j\bigl(T_j-3(N-1)\bigr)
\ge\tfrac18T_1T_2T_3=4096N^4>(4N-3)^4.
\]
Thus, \eqref{eq:degcond} holds, and by
Proposition~\ref{prop:interp} every $f\in\C^{\cube}$ has the form
$f(\vn)=\ipol(\lmap\vn)$ for some $\ipol\in{\tt Poly}(\vfd)$. The
monomials $z_1^{\alpha_1}z_2^{\alpha_2}z_3^{\alpha_3}$ with
$\valpha\in\Z_{\ge 0}^3$ and $\alpha_j\le\fd_j$ form a basis of
${\tt Poly}(\vfd)$, so $f$ is a linear combination of the vectors
in the statement with $\alpha_j\le\fd_j$.
\end{proof}

\section{The field, the points, and the weighted determinant}
\label{sec:determinant}

\subsection{The biquadratic ring}

Assume that $q\ge 9$, and let $d\ne 1,2$ be the squarefree integer
of \eqref{eq:field}. Fix $a,b\in\C$ with $a^2=d$, $b^2=2$, and put
\[
K\defeq\Q(a,b),\qquad\cO\defeq\Z[a,b]
=\Z\oplus\Z a\oplus\Z b\oplus\Z ab.
\]
Since $d$, $2$, and $2d$ are not rational squares, $[K:\Q]=4$ and
$1,a,b,ab$ is a $\Q$-basis of $K$; in particular the expression of an
element of $\cO$ in this $\Z$-basis is unique. The extension $K/\Q$ is
Galois with group $\{1,\sigma,\tau,\sigma\tau\}$, where
\[
\sigma(a)=-a,\quad\sigma(b)=b;\qquad\tau(a)=a,\quad\tau(b)=-b.
\]
The four field embeddings $K\hookrightarrow\C$ are the four elements
of the group. For $z\in K$ let
$\Nm(z)\defeq z\,\sigma(z)\,\tau(z)\,\sigma\tau(z)$.
Clearly, $\Nm(\cO)\subset\Z$.

\subsection{The points}

Fix integers $1\le H\le N$ and set
\[
M\defeq N^{4/3},\qquad\text{so that}\quad M^3=N^4=\bigl|\cube\bigr|.
\]
Define
\be\label{eq:thevndefn}
\theta_{\vn}\defeq n_1+n_2a+n_3b+n_4ab\in\cO\qquad(\vn\in \cube).
\ee
Then
\[
\bigl(\theta_{\vn},\sigma(\theta_{\vn}),\sigma\tau(\theta_{\vn})\bigr)=\lmap\vn,
\qquad\text{where}\quad
\lmap\defeq\begin{pmatrix}
1 & a & b & ab\\
1 & -a & b & -ab\\
1 & -a & -b & ab
\end{pmatrix}.
\]
The first three columns of $\lmap$ form a matrix of determinant
$4ab\ne 0$, so $\lmap$ is surjective, and one checks directly that
$\lmap\vxi=\vzero$ for $\vxi\defeq(ab,b,-a,-1)$.
Thus, $\ker\lmap=\C\vxi$, and the coordinates
of $\vxi$ are linearly independent over $\Q$, so
Proposition~\ref{prop:interp} and
Corollary~\ref{cor:box} can be applied to $\lmap$.

\subsection{Rows, weights, and the greedy selection}

For $\valpha\in\Z_{\ge 0}^3$ we define the row vector
\[
\row_{\valpha}\defeq\bigl((\lmap\vn)^{\valpha}\bigr)_{\vn\in\cube}
=\bigl(\theta_{\vn}^{\alpha_1}\,\sigma(\theta_{\vn})^{\alpha_2}\,
\sigma\tau(\theta_{\vn})^{\alpha_3}\bigr)_{\vn\in \cube}\in\cO^{\cube}.
\]
These are precisely the vectors of Corollary~\ref{cor:box} for the
present map $\lmap$. We also define the weight
\[
w(\valpha)\defeq\alpha_1+H\alpha_2+H\alpha_3.
\]
Fix a total order $\prec$ on $\Z_{\ge 0}^3$ such that
$w(\vbeta)<w(\valpha)$ implies $\vbeta\prec\valpha$, and ties are broken
by any fixed rule. Each weight is attained finitely often, so this
order has the same order type as $\N$.

\begin{definition}\label{def:greedy}
Progress through the indices in the order $\prec$ and retain $\valpha$ if
and only if $\row_{\valpha}$ is not in the $K$-span of
$\{\row_{\vbeta}:\vbeta\prec\valpha\}$.
Let $\cA$ denote the set of retained indices.
\end{definition}

\begin{lemma}\label{lem:greedy}
The following statements hold.
\begin{itemize}
\item[$(a)$] For every $\valpha\in\Z_{\ge 0}^3$ (retained or not),
\[
\spn_K\{\row_{\vbeta}:\vbeta\in\Z_{\ge 0}^3,\ \vbeta\preceq\valpha\}
=\spn_K\{\row_{\vbeta}:\vbeta\in\cA,\ \vbeta\preceq\valpha\}.
\]
\item[$(b)$] The rows $\row_{\valpha}$ with $\valpha\in\cA$ form a
$K$-basis of $K^{\cube}$, and thus $|\cA|=N^4$. Moreover, every
$\valpha\in\cA$ satisfies
\[
w(\valpha)\le 96H^{2/3}M
\mand
\alpha_2,\alpha_3\le 96H^{-1/3}M.
\]
\end{itemize}
\end{lemma}

\begin{proof}
Part $(a)$ follows by induction along $\prec$, since a row that is
not retained lies in the $K$-span of the earlier rows.

For $(b)$, the retained rows are linearly independent over $K$ by
Definition~\ref{def:greedy}, so $|\cA|\le N^4$. Let $T_1,T_2,T_3$ be
as in Corollary~\ref{cor:box}. By that corollary, the rows
$\row_{\valpha}$ with $\alpha_j\le T_j$ for each $j$ span $\C^{\cube}$,
hence also $K^{\cube}$, since their entries lie in $K$. Each of them
satisfies
\[
w(\valpha)\le T_1+HT_2+HT_3=96H^{2/3}M.
\]
By the choice of $\prec$, the indices of weight at most
$96H^{2/3}M$ form an initial segment of the order, so by $(a)$ the
retained rows with such indices span $K^{\cube}$. Hence there are
$N^4$ of them, and no index of larger weight is retained. The bounds on
$\alpha_2$ and $\alpha_3$ follow from $H\alpha_2\le w(\valpha)$ and
$H\alpha_3\le w(\valpha)$.
\end{proof}

List $\cA$ in the order $\prec$, fix an ordering of $\cube$ for the
columns, and define
\[
\Delta\defeq\det\bigl(\row_{\valpha}\bigr)_{\valpha\in\cA}\in\cO\setminus\{0\},
\qquad S_1\defeq\sum_{\valpha\in\cA}\alpha_1,\qquad
S_{23}\defeq\sum_{\valpha\in\cA}(\alpha_2+\alpha_3).
\]
Note that $\Delta$ is formed from the monomial rows $\row_{\valpha}$
alone. The auxiliary polynomials in the proof of
Proposition~\ref{prop:interp} served only to show that these rows
span $K^{\cube}$, and no bound on their coefficients is needed in what
follows.

\begin{lemma}[Most of the degree is in the first variable]
\label{lem:degree}
If
\[
1\le H\le N
\mand
H^{2/3}N^{4/3}\ge 2\cdot 97^2,
\] then
\[
S_1\ge (4\cdot 97^2)^{-1}N^4H^{2/3}M\mand S_{23}\le 192N^4H^{-1/3}M.
\]
\end{lemma}

\begin{proof}
Let
\[
B\defeq\lfloor 96H^{-1/3}M\rfloor+1.
\]
By Lemma~\ref{lem:greedy}$(b)$, the integers $\alpha_2,\alpha_3$
satisfy the bound $\alpha_2,\alpha_3\le B-1$ for all $\valpha\in\cA$, and
thus $S_{23}\le 2(B-1)N^4\le 192N^4H^{-1/3}M$.

For each value of $\alpha_1$ there are at most $B^2$ retained
indices, one for each pair $(\alpha_2,\alpha_3)$ with
$0\le\alpha_2,\alpha_3\le B-1$. List
the first coordinates of the retained indices in increasing order as
$i_1\le\cdots\le i_{N^4}$. At least $j$ retained indices have first
coordinate at most $i_j$, so $j\le(i_j+1)B^2$, and hence
$i_j\ge\lfloor(j-1)/B^2\rfloor$.
Writing $N^4=B^2k+\varrho$ with $0\le \varrho<B^2$,
\[
S_1\ge\sum_{j=1}^{N^4}\Bigl\lfloor\frac{j-1}{B^2}\Bigr\rfloor
=B^2\frac{k(k-1)}{2}+k\varrho
=\frac{N^8}{2B^2}-\frac{N^4}{2}+\frac{\varrho(B^2-\varrho)}{2B^2}
\ge\frac{N^8}{2B^2}-\frac{N^4}{2}.
\]
Since $H\le N$, we have $H^{-1/3}M\ge N\ge 1$, and therefore
\[
B\le 96H^{-1/3}M+1\le 97H^{-1/3}M.
\]
Together with the hypothesis $H^{2/3}N^{4/3}\ge 2\cdot 97^2$, this gives
\[
2B^2\le 2\cdot 97^2H^{-2/3}M^2\le M^3=N^4,
\]
so that $\tfrac{N^4}{2}\le\tfrac{N^8}{4B^2}$. Consequently,
\[
S_1\ge\frac{N^8}{4B^2}\ge\frac{N^8}{4\cdot 97^2H^{-2/3}M^2}
=(4\cdot 97^2)^{-1}N^4H^{2/3}M.\qedhere
\]
\end{proof}

\section{The Archimedean bound}
\label{sec:arch}

We bound $\tfrac14\log|\Nm(\Delta)|$, the average of $\log|g(\Delta)|$
over the four embeddings $g$ of~$K$, from above in this section
and from below in the next.

\begin{lemma}\label{lem:arch}
With $\sL\defeq\log q$ as before, we have
\be\label{eq:arch}
\frac14\log|\Nm(\Delta)|\le 2N^4\log N
+(S_1+S_{23})\Bigl(\log N+\frac{\sL}{2}+\log 8\Bigr).
\ee
\end{lemma}

\begin{proof}
Let $g\in\Gal(K/\Q)=\{1,\sigma,\tau,\sigma\tau\}$. Recalling the
definition \eqref{eq:thevndefn}, and taking into account
the bounds $|a|=\sqrt{|d|}$, $|b|=\sqrt{2}$, and $1\le|d|\le q$,
for every $\vn\in \cube$ we have
\[
\bigl|g(\theta_{\vn})\bigr|\le(N-1)\bigl(1+|a|+|b|+|ab|\bigr)
=(N-1)\bigl(1+\sqrt{|d|}\bigr)\bigl(1+\sqrt{2}\bigr)<8N\sqrt{q}.
\]
Now $g(\Delta)$ is the determinant of the matrix with rows
$g(\row_{\valpha})$, whose entries are
\[
g(\theta_{\vn})^{\alpha_1}\,g\sigma(\theta_{\vn})^{\alpha_2}\,
g\sigma\tau(\theta_{\vn})^{\alpha_3}.
\]
Each factor is the image of $\theta_{\vn}$ under some element of
$\Gal(K/\Q)$, so each entry has absolute value not exceeding
$(8N\sqrt{q})^{|\valpha|}$, where
$|\valpha|\defeq\alpha_1+\alpha_2+\alpha_3$.
Consequently, the row $g(\row_{\valpha})$ has Euclidean norm
at most $N^2\,(8N\sqrt{q})^{|\valpha|}$. By Hadamard's inequality,
\[
|g(\Delta)|\le\prod_{\valpha\in\cA}N^2\,(8N\sqrt{q})^{|\valpha|},
\]
or equivalently,
\[
\log|g(\Delta)|\le 2N^4\log N+(S_1+S_{23})\log(8N\sqrt{q}\,).
\]
This bound holds for each $g\in\Gal(K/\Q)$.
Since $|\Nm(\Delta)|=\prod_g|g(\Delta)|$, averaging the preceding bound
over $g$ we obtain \eqref{eq:arch}.
\end{proof}

\section{The \texorpdfstring{$p$}{p}-adic bound}
\label{sec:padic}

\begin{definition}\label{def:admissible}
An odd prime $p$ is \emph{admissible} if $p>H$ and $\chi(p)=-1$.
\end{definition}

\begin{lemma}[Frobenius congruence]\label{lem:frob}
Let $p$ be an odd prime with $\chi(p)=-1$, and put
\[
g_p\defeq\begin{cases}
\sigma&\text{if $\smallleg{2}{p}=1$,}\\[1ex]
\sigma\tau&\text{if $\smallleg{2}{p}=-1$.}
\end{cases}
\]
Then
\[
\theta^p\equiv g_p(\theta)\bmod{p\,\cO}\qquad(\theta\in\cO).
\]
\end{lemma}

\begin{proof}
The ring $\cO/p\,\cO$ has characteristic $p$, so its Frobenius
endomorphism $x\mapsto x^p$ is a ring homomorphism, and it fixes the
image of $\Z$ by Fermat's little theorem.
By Euler's criterion, and since $\chi(p)=-1$ in \eqref{eq:field},
\[
a^p=a\cdot d^{(p-1)/2}\equiv-a\bmod{p\,\cO}
\mand
b^p=b\cdot 2^{(p-1)/2}\equiv\leg{2}{p}b\bmod{p\,\cO}.
\]
Hence, for every $\theta=z_1+z_2a+z_3b+z_4ab$ with
$z_1,\ldots,z_4\in\Z$, we have
\[
\theta^p\equiv z_1-z_2a+\leg{2}{p}z_3b-\leg{2}{p}z_4ab
\bmod{p\,\cO},
\]
and the right side is $g_p(\theta)$ by the definition of $g_p$.
\end{proof}

In other words, $g_p$ is the Frobenius element of $\Gal(K/\Q)$ at $p$.

\begin{lemma}[Extraction of prime powers]\label{lem:extract}
For every admissible prime $p$,
\[
\Delta\in p^{E_p}\cO,\qquad
E_p\defeq\sum_{\valpha\in\cA}\Bigl\lfloor\frac{\alpha_1}{p}\Bigr\rfloor,
\]
and consequently $p^{4E_p}$ divides the nonzero integer
$\Nm(\Delta)$.
\end{lemma}

\begin{proof}
Let $j\defeq 2$ if $g_p=\sigma$ and $j\defeq 3$ if $g_p=\sigma\tau$,
so that $g_p(\theta_{\vn})$ is the $j$-th coordinate of $\lmap\vn$. Let
$\ve_1,\ve_2,\ve_3$ be the standard basis of $\Z^3$. For
$\valpha\in\cA$ write $\alpha_1=pk+\varrho$ with $0\le\varrho<p$,
$k=\lfloor\alpha_1/p\rfloor$, and define a new row
\[
\widetilde{\row}_{\valpha}\defeq\Bigl(\theta_{\vn}^{\varrho}\,
\bigl(\theta_{\vn}^p-g_p(\theta_{\vn})\bigr)^k\,
\sigma(\theta_{\vn})^{\alpha_2}\,\sigma\tau(\theta_{\vn})^{\alpha_3}
\Bigr)_{\vn\in \cube}.
\]

\medskip\noindent{\sc Step 1: \it divisibility.} By Lemma~\ref{lem:frob},
$\theta_{\vn}^p-g_p(\theta_{\vn})\in p\,\cO$, so every entry of
$\widetilde{\row}_{\valpha}$ lies in the module $p^k\cO$.

\medskip\noindent{\sc Step 2: \it transforming the rows.}
We claim that $\widetilde{\row}_{\valpha}$ is the sum of $\row_{\valpha}$
and some earlier rows (relative to $\prec$). Expanding by the binomial
theorem, we have
\[
\theta_{\vn}^{\varrho}\bigl(\theta_{\vn}^p-g_p(\theta_{\vn})\bigr)^k
=\sum_{m=0}^{k}(-1)^m\binom{k}{m}
\theta_{\vn}^{\alpha_1-pm}g_p(\theta_{\vn})^m,
\]
hence, exactly,
\[
\widetilde{\row}_{\valpha}=\row_{\valpha}
+\sum_{m=1}^{k}(-1)^m\binom{k}{m}\,\row_{\valpha-pm\,\ve_1+m\,\ve_j}.
\]
The indices on the right lie in $\Z_{\ge 0}^3$ because for each
$1\le m\le k$, we have
\[
\alpha_1-pm\ge\varrho\ge 0,
\]
and because $p>H$,
\[
w(\valpha-pm\,\ve_1+m\,\ve_j)=w(\valpha)-pm+Hm=w(\valpha)-m(p-H)<w(\valpha).
\]
Thus, every such index $\valpha-pm\,\ve_1+m\,\ve_j$
strictly precedes $\valpha$ in the order $\prec$. By
Lemma~\ref{lem:greedy}$(a)$, the corresponding row lies in
$\spn_K\{\row_{\vbeta}:\vbeta\in\cA,\ \vbeta\prec\valpha\}$.
Therefore
\[
\widetilde{\row}_{\valpha}=\row_{\valpha}
+\ssum{\vbeta\in\cA\\ \vbeta\prec\valpha}L_{\valpha\vbeta}\,\row_{\vbeta}
\qquad(L_{\valpha\vbeta}\in K).
\]

\medskip\noindent{\sc Step 3: \it conclusion.}
With $\cA$ ordered by $\prec$, the argument in Step~2
says that the matrix $(\widetilde{\row}_{\valpha})_{\valpha\in\cA}$ is
equal to $L\cdot(\row_{\valpha})_{\valpha\in\cA}$ with some lower triangular matrix $L$ that has ones on the diagonal (in other words, a lower unipotent matrix). Hence
\[
\det\bigl(\widetilde{\row}_{\valpha}\bigr)_{\valpha\in\cA}
=\det\bigl(\row_{\valpha}\bigr)_{\valpha\in\cA}=\Delta.
\]
By Step~1, each $\widetilde{\row}_{\valpha}$ is
$p^{\lfloor\alpha_1/p\rfloor}$
times a row with entries in $\cO$; pulling these factors out of the
determinant gives $\Delta=p^{E_p}x$ with $x\in\cO$. Finally,
$\Nm(\Delta)=p^{4E_p}\Nm(x)$ with $\Nm(x)\in\Z$.
\end{proof}

\begin{corollary}\label{cor:padic}
Let $C_1$ and $C_2$ be as in Lemma~\ref{lem:bias} for some
$\lambda\in(0,2]$. Let $\beta\in(0,1)$ be a real zero of
$L(s,\chi)$, $\delta=(1-\beta)\sL$, $H\ge 319$, and $M\ge 8$. Then
\be\label{eq:padic}
\frac14\log|\Nm(\Delta)|\ge S_1\Bigl(\log M-\log H
-C_1\,\frac{\delta(\log M)^2}{\sL}
-\Bigl(C_2+\frac{1}{2H}\Bigr)\sL\Bigr)-2N^4M.
\ee
\end{corollary}

\begin{proof}
By Lemma~\ref{lem:extract}, the nonzero integer $\Nm(\Delta)$ is
divisible by $p^{4E_p}$ for every admissible prime $p$, and hence
by the product of these prime powers over the admissible primes
$p\le M$. Therefore,
\[
\frac14\log|\Nm(\Delta)|\ge\ssum{p\le M\\ p\text{ admissible}}E_p\log
p.
\]
Since $\lfloor u\rfloor>u-1$ and $|\cA|=N^4$, we have
$E_p\ge S_1/p-N^4$ for any such $p$, and so
\[
\frac14\log|\Nm(\Delta)|
\ge S_1\ssum{p\le M\\ p\text{ admissible}}\frac{\log p}{p}
-N^4\sum_{p\le M}\log p.
\]
As $H\ge 319$, the admissible primes $p\le M$ are precisely the primes
$H<p\le M$ with $\chi(p)=-1$. Hence \eqref{eq:bias2} with
$X\defeq M\ge 8$ shows that the first sum is at least the expression
in parentheses in \eqref{eq:padic}. Since $S_1\ge 0$, and since
$\sum_{p\le M}\log p\le 2M$ by Chebyshev's bound, the result follows.
\end{proof}

\section{Conclusion of the proof}
\label{sec:conclusion}

\begin{proof}[Proof of Theorem~\ref{thm:main}]
We apply Lemma~\ref{lem:bias} with $\lambda\defeq\tfrac32$ and with
\[
C_1\defeq\frac{2\er^{3/2}}{9}\mand C_2\defeq\frac{\er^{3/2}}{4},
\]
and we put
\[
H\defeq 9216\cdot 97^2,\qquad
\omega\defeq\lceil 8(C_2+\tfrac85)\rceil,\qquad
q_0\defeq10^{10},\qquad
\delta_0\defeq\frac{9}{128(\omega+1)C_1}.
\]
We show that the first assertion of the theorem holds with these
values of $q_0$ and $\delta_0$. Let $\chi$ have conductor $q\ge q_0$,
let $\beta\in(0,1)$ be a real zero of $L(s,\chi)$, and put
$\delta\defeq(1-\beta)\sL$.

\medskip\noindent{\sc Step 1: \it the parameters.}
Let $N\defeq\lceil q^\omega\rceil$ and $M\defeq N^{4/3}$. Since
$q^\omega\le N\le q^{\omega+1}$, we have
\be\label{eq:logM}
\frac43\,\omega\,\sL\le\log M\le\frac43(\omega+1)\sL.
\ee
Moreover, $q\ge q_0\ge H\ge 9$, hence $d\ne 2$, and
\[
319\le H\le q\le N\mand M\ge 8.
\]
With these values of $H$ and $N$ we form $\Delta$, $S_1$, and
$S_{23}$ as in \S\ref{sec:determinant}. Lemma~\ref{lem:arch} and
Corollary~\ref{cor:padic} can then be applied.

\medskip\noindent{\sc Step 2: \it the sums $S_1$ and $S_{23}$.}
Since $H\ge 8\cdot 97^3$, we have $H^{2/3}\ge 4\cdot 97^2$. In
particular, the hypotheses of Lemma~\ref{lem:degree} are met, and
that lemma gives
\[
S_1\ge(4\cdot 97^2)^{-1}N^4H^{2/3}M\ge N^4M
\mand
S_{23}\le 192N^4H^{-1/3}M\le\frac{768\cdot 97^2}{H}\,S_1.
\]
By the choice of $H$, it follows that
\be\label{eq:S1S23}
S_1\ge N^4M\mand S_{23}\le\frac{1}{12}\,S_1.
\ee

\medskip\noindent{\sc Step 3: \it comparison of the two bounds.}
Combining \eqref{eq:arch} and \eqref{eq:padic}, and using the relation
$\log N=\tfrac34\log M$, we get that
\dalign{
&S_1\Bigl(\log M-\log H
-C_1\,\frac{\delta(\log M)^2}{\sL}
-\Bigl(C_2+\frac{1}{2H}\Bigr)\sL\Bigr)-2N^4M\\
&\qquad\qquad\qquad\qquad\le\frac32N^4\log M
+(S_1+S_{23})\Bigl(\frac34\log M+\frac{\sL}{2}+\log 8\Bigr).
}
By \eqref{eq:S1S23}, we have
\[
2N^4M\le 2S_1,\qquad
\frac32N^4\log M\le\frac{3\log M}{2M}\,S_1\le S_1,\qquad
S_1+S_{23}\le\frac{13}{12}\,S_1.
\]
Inserting these bounds and dividing by $S_1>0$, we obtain
\[
\frac{3}{16}\log M\le C_1\,\frac{\delta(\log M)^2}{\sL}
+\Bigl(C_2+\frac{1}{2H}+\frac{13}{24}\Bigr)\sL
+\log H+3+\frac{13}{12}\log 8.
\]
Since $q\ge q_0$, we have $\sL>23$, and hence
$\log H+3+\tfrac{13}{12}\log 8<24<\tfrac{24}{23}\sL$. Moreover,
$\tfrac{1}{2H}+\tfrac{13}{24}+\tfrac{24}{23}<\tfrac85$. Therefore,
\be\label{eq:final}
\frac{3}{16}\log M\le C_1\,\frac{\delta(\log M)^2}{\sL}
+(C_2+\tfrac85)\sL.
\ee

\medskip\noindent{\sc Step 4: \it conclusion.}
The inequality $\delta\ge\delta_0$ is equivalent to
$\tfrac{3}{16}\le\tfrac43(\omega+1)C_1\,\delta+\tfrac{3}{32}$ by the
definition of $\delta_0$. To prove the latter, we divide
\eqref{eq:final} by $\log M$ and use \eqref{eq:logM}, and then the
choice of $\omega$ in the last step. This gives
\[
\frac{3}{16}\le C_1\,\delta\,\frac{\log M}{\sL}
+(C_2+\tfrac85)\frac{\sL}{\log M}
\le\frac43(\omega+1)C_1\,\delta
+\frac{3(C_2+\frac85)}{4\omega}
\le\frac43(\omega+1)C_1\,\delta+\frac{3}{32}.\qedhere
\]
\end{proof}

\begin{remark}[Effectivity]\label{rem:effective}
The constants $q_0$ and $\delta_0$ in Theorem~\ref{thm:main}
are explicit. The proof of the theorem applies
Lemma~\ref{lem:bias} with $\lambda=\tfrac32$, for which
\[
C_1=\frac{2\er^{3/2}}{9}=0.9959\cdots\mand
C_2=\frac{\er^{3/2}}{4}=1.1204\cdots,
\]
and the lemma in turn relies on the explicit form of Mertens'
theorem given by Rosser and Schoenfeld~\cite[Theorem~6]{RS}. With
these values, the proof of Theorem~\ref{thm:main} has
$\omega=22$,
\[
q_0=10^{10}\mand
\delta_0=\frac{81}{5888\,\er^{3/2}}=3.069\cdots\times 10^{-3},
\]
so the first assertion of the theorem holds with
\[
q_0=10^{10}\mand\delta_0=3.0\times 10^{-3}.
\]
The constant $c$ in the second assertion depends also on the real
zeros of $L(s,\chi)$ for the finitely many characters $\chi$ of
conductor $q<q_0$. As shown in the proof of the second
assertion, a computation that we have not verified bounds
these zeros away from one, and with it one can take $c=3.0\times 10^{-3}$.
\end{remark}

\begin{remark}[Where each hypothesis is used]\label{rem:hypotheses}
\leavevmode
\setlength{\leftmargini}{2.6em}
\renewcommand{\labelenumi}{\textup{(\roman{enumi})}}
\begin{enumerate}
\item The zero $\beta$ enters only through Lemma~\ref{lem:bias}. By
\eqref{eq:bias2}, it guarantees that the admissible primes $p\le M$
satisfy
\[
\sum_{p}\frac{\log p}{p}\ge\log M-\log H
-C_1\,\frac{\delta(\log M)^2}{\sL}
-\Bigl(C_2+\frac{1}{2H}\Bigr)\sL,
\]
which is close to $\log M$ when $\delta$ is small and $\log M$ is a
large multiple of $\sL$.
\item The condition $\chi(p)=-1$ ensures that the Frobenius element
at $p$ is one of the two automorphisms $\sigma$ and $\sigma\tau$
that appear among the variables. A prime with Frobenius element
$\tau$, that is, with $\chi(p)=1$ and $\smallleg{2}{p}=-1$, cannot
be used, since $\tau(\theta)$ is not one of the three variables
$\theta$, $\sigma(\theta)$, $\sigma\tau(\theta)$.
\item The condition $p>H$ ensures that replacing $\theta^p$ by
$g_p(\theta)$ decreases the weight. The weight $H$ on the second and
third variables gives $S_{23}\le\tfrac{1}{12}S_1$ for our choice of
$H$, so that the main term $\tfrac34(S_1+S_{23})\log M$ of the upper
bound is at most $\tfrac{13}{16}S_1\log M$, which is less than the
main term $S_1\log M$ of the lower bound.
\item The auxiliary number $\sqrt{2}$ raises the rank of the lattice
to four while keeping three variables. This is the source of the
relation $\log N=\tfrac34\log M$.
\item The linear independence over $\Q$ of the coordinates of
$\vxi$, which amounts to that of $1,a,b,ab$, allows
Proposition~\ref{prop:interp} to be applied with constants that do
not depend on the number $d$.
\item The choice $N=\lceil q^\omega\rceil$ with $\omega$ large makes
the terms of size $\sL$, which come from the size of $\theta_{\vn}$
and from the loss $\bigl(C_2+\tfrac{1}{2H}\bigr)\sL$ in
Lemma~\ref{lem:bias}, at most
$\tfrac{3}{32}\log M$.
\item The condition $q\ge q_0$ serves to absorb the bounded terms
$\log H+3+\tfrac{13}{12}\log 8$ into a multiple of
$\sL$. It also gives $q\ge 9$, hence $d\ne 2$,
and $N\ge H$.
\end{enumerate}
\end{remark}

\section*{Acknowledgments}

The author thanks the Max Planck Institute for Mathematics in
Bonn for its support and excellent working conditions during the
preparation of this paper.

The argument presented in this paper is due to OpenAI~\cite{OAI}.
The large language model Claude (Anthropic) was used throughout the
preparation of the paper. It produced a report on the
manuscript~\cite{OAI} together with an expanded draft of the proof,
which is the basis of the present text. It was also used for
drafting and revising portions of the text and of the proofs, for
working out the explicit constants, for preparing the figure and
the diagrams, for literature searches, and for proof checking. The
author has worked through every line of the paper, has verified the
mathematical arguments in their entirety, and takes full
responsibility for its content.


\begin{thebibliography}{99}

\bibitem{Dav}
H.~Davenport,
\emph{Multiplicative number theory}, 2nd ed., revised by
H.~L.~Montgomery.
Graduate Texts in Mathematics, vol.~74, Springer-Verlag, New York,
1980.

\bibitem{Gol}
D.~M.~Goldfeld,
The class number of quadratic fields and the conjectures of Birch
and Swinnerton-Dyer.
\emph{Ann. Scuola Norm. Sup. Pisa Cl. Sci. (4)} 3 (1976), no.~4,
623--663.

\bibitem{GZ}
B.~H.~Gross and D.~B.~Zagier,
Heegner points and derivatives of $L$-series.
\emph{Invent. Math.} 84 (1986), no.~2, 225--320.

\bibitem{Lau}
M.~Laurent,
Sur quelques r\'esultats r\'ecents de transcendance.
\emph{Ast\'erisque} 198--199--200 (1991), 209--230.

\bibitem{LZZ}
R.~F.~Lu, A.~Zaman, and H.~Zhao,
Numerical computations concerning Landau--Siegel zeros.
Preprint, arXiv:2602.03626, 2026.

\bibitem{OAI}
OpenAI,
Uniform exclusion of Landau--Siegel zeros.
Manuscript, dated October 1, 2026.

\bibitem{Pag}
A.~Page,
On the number of primes in an arithmetic progression.
\emph{Proc. London Math. Soc. (2)} 39 (1935), 116--141.

\bibitem{RS}
J.~B.~Rosser and L.~Schoenfeld,
Approximate formulas for some functions of prime numbers.
\emph{Illinois J. Math.} 6 (1962), 64--94.

\bibitem{Sie}
C.~L.~Siegel,
\"Uber die Classenzahl quadratischer Zahlk\"orper.
\emph{Acta Arith.} 1 (1935), 83--86.

\end{thebibliography}
\end{document}